\documentclass[a4paper,11pt,oneside,reqno]{amsart}

\usepackage{amsmath,amssymb,amsthm,mathtools}
\usepackage[pagewise]{lineno}
\mathtoolsset{showonlyrefs}
\usepackage{bm}
\usepackage{subfigure}
\usepackage{verbatim}
\usepackage{wrapfig}
\usepackage{ascmac}
\usepackage{cases}
\usepackage{autobreak}
\usepackage{comment}
\usepackage{cancel}
\usepackage[normalem]{ulem}
\usepackage{xcolor}
\usepackage{marginnote}
\usepackage{mparhack}
\usepackage[hidelinks]{hyperref}

\newif\ifshowcomments
\showcommentstrue

\makeatletter

\@addtoreset{equation}{section}
\makeatother

\newtheorem{theorem}{Theorem}[section]

\newtheorem{lemma}[theorem]{Lemma}

\newtheorem{assumption}[theorem]{Assumption}
\newtheorem{remark}[theorem]{Remark}
\newtheorem{definition}[theorem]{Definition}

\newtheorem{example}[theorem]{Example}

\newcounter{c}
\newcounter{d}
\newcounter{b}

\newcommand{\R}{\mathbb{R}}

\newcommand{\e}{\varepsilon}

\newcommand{\supp}{\mathrm{supp}}

\newcommand{\diver}{\operatorname{div}}

\newcommand{\dual}[2]{\left\langle #1,#2\right\rangle}
\newcommand{\M}{\mathcal{M}}
\newcommand{\KM}{\mathfrak{M}}

\begin{document}

\title[Delayed diffusion with measure-valued kernels]{Delayed diffusion with measure-valued kernels in nonlinear parabolic equations}
\author{Yuki Tsukamoto}
\date{}

\address{Tokyo University of Science, 162-8601, Tokyo, Japan}
\email{tsukamoto-yuki@rs.tus.ac.jp}

\begin{abstract}
We study nonlinear parabolic equations with delayed diffusion terms governed
by finite signed measure kernels. The atom of the kernel at the origin is
absorbed into the present-time operator, while the remaining part is treated
as a residual delay kernel. Under structural assumptions on the effective
present-time operators and a pathwise coercivity condition for the total
memory operator, we prove the existence and uniqueness of weak solutions and
their stability under weak-star convergence of the kernels. The stability
result covers collapsing delayed atoms, whose mass is transferred to the
present-time diffusion coefficient in the limit. We verify the assumptions
for \(p\)-Laplacian type examples, including separated kernels and a
regularized class of kernels reaching the origin.
\end{abstract}

\maketitle

\section{Introduction and main results}

Nonlinear parabolic equations with delay arise naturally in models where the
diffusive dynamics depend not only on the present state of the system but
also on its history. A prototype for the equations studied here is
\[
  u'(t)
  +a\mathcal C_p(u(t))
  +b\mathcal C_p(u(t-\tau))
  =f(t),
  \qquad
  \mathcal C_p(u)
  :=
  -\operatorname{div}
  \bigl(|\nabla u|^{p-2}\nabla u\bigr).
\]
Here the nonlinear diffusion operator itself is evaluated at the delayed
time \(t-\tau\). In the general problem considered below, this discrete
delayed contribution is replaced by a measure-valued term of the form
\[
  \int_{[0,r]}
  \mathcal C(\bar u(t-s))\,d\mu(s),
\]
where \(\bar u\) denotes the extension of \(u\) by its prescribed past
history, \(\mathcal C\) is a nonlinear spatial operator, and \(\mu\) is a
finite signed measure on the delay interval \([0,r]\). Taking
\(\mathcal C=\mathcal C_p\) and \(\mu=b\delta_\tau\) recovers the delayed
term in the prototype above, whereas an absolutely continuous measure yields
a distributed delay.

Delayed diffusion is analytically different from delay in a reaction or
source term. The usual monotonicity inequality for a nonlinear diffusion
operator pairs the difference of two operator values with the difference of
the corresponding states at the same time. When two solutions of a delayed
diffusion equation are compared, however, one encounters terms of the form
\[
  \left\langle
    \mathcal C(\bar u(t-s))-\mathcal C(\bar v(t-s)),
    u(t)-v(t)
  \right\rangle.
\]
The operator difference and the state difference are evaluated at different
times, so pointwise monotonicity of \(\mathcal C\) does not by itself provide
a useful sign. The main analytical issue is therefore to combine the
dissipation of the present-time operator with the contribution from the past
and obtain a comparison estimate at the level of entire trajectories.

Delay differential equations have a classical theory within the framework
of functional differential equations, for which the monograph by Hale and
Verduyn Lunel is a standard reference
\cite{HaleVerduynLunel1993IntroductionFunctionalDifferentialEquations}.
Equations with spatial variables are treated in the theory of partial
functional differential equations; see, for example, Wu
\cite{Wu1996TheoryApplicationsPartialFunctionalDifferentialEquations}.
Abstract evolution equations with Volterra-type memory are discussed by
Pr\"uss
\cite{Pruss1993EvolutionaryIntegralEquationsApplications}.
Quasilinear parabolic equations with delays in the highest-order spatial
derivatives were studied by Yong and Pan
\cite{YongPan1993QuasilinearParabolicDelaysHighestOrder}
in a framework allowing a possibly countable family of discrete delays and
finitely many distributed delays.

Delayed reaction-diffusion equations have been studied extensively,
particularly in connection with travelling waves. Representative results
include travelling waves for delayed reaction-diffusion systems
\cite{BoumenirNguyen2008PerronMonotoneIterationTravelingWaves,
TianLin2010TravelingWavesDelayedReactionDiffusionSystems},
positive travelling fronts for systems with distributed delay
\cite{FariaTrofimchuk2010PositiveTravellingFrontsDistributedDelay},
and travelling waves for bistable equations with delay
\cite{TrofimchukVolpert2018TravelingWavesBistableDelay}.
In much of this literature, however, the delay enters through a lower-order
reaction term, while the diffusion operator acts on the present state. 
This is also the case in the delayed \(p\)-Laplacian population model
studied in
\cite{YangDeng2017QualitativePropertiesPLaplacianPopulationModelDelay}:
the \(p\)-Laplacian acts on the present state, while the delay enters through
the population growth term. Barker and Minh have studied travelling waves
for equations with delay in both diffusion and reaction, but their diffusion
operator is linear
\cite{BarkerMinh2026TravelingWavesDelayedDiffusionReaction}.

Antontsev et al. studied evolution problems for the \(p\)- and
\(p(x)\)-Laplacians with Volterra-type memory
\cite{AntontsevShmarevSimsenSimsen2016EvolutionPLaplacianMemory,
AntontsevShmarevSimsenSimsen2019DifferentialInclusionPXMemory}.
These works assume regular Volterra kernel functions and do not cover
general finite signed measures, including singular and atomic kernels,
stability under weak-star convergence of the total kernels, or limits
in which delayed atoms collapse to the origin. Thus, the corresponding
measure-valued theory for nonlinear diffusion in the principal part
remains much less developed.

By contrast, nonlinear parabolic equations without delay have a broad and
well-established theory. Abstract evolution equations governed by monotone
operators are treated systematically in
\cite{Barbu2010NonlinearDifferentialEquationsMonotoneTypesBanachSpaces,
Showalter1997MonotoneOperatorsBanachSpaceNonlinearPDE}.
For the foundational theory of degenerate parabolic equations of
\(p\)-Laplacian type, we refer to
\cite{DiBenedetto1993DegenerateParabolicEquations}.
A substantial regularity theory is also available; see, for example,
\cite{Boegelein2014GlobalCalderonZygmundNonlinearParabolicSystems,
CianchiMazya2020SecondOrderRegularityParabolicPLaplaceProblems}.
Equations involving measure data are studied, for example, in
\cite{ByunParkShin2021GlobalRegularityDegenerateSingularParabolicMeasureData}.
Asymptotic behaviour and long-time dynamics are considered in
\cite{CruzUribeMoenShao2025DegenerateParabolicPLaplacianEquations} and
\cite{FolinoPlazaStrani2021LongTimeDynamicsPLaplacianBistable},
respectively. The standard monotonicity-based well-posedness theory,
however, does not directly extend to nonlinear diffusion operators evaluated
along the past history.

Measure-valued kernels provide a common framework for discrete and
distributed delays. Weak-star convergence also provides a natural setting
for limits in which the weights and locations of atomic delays may vary.
Continuous dependence on coefficients and delay parameters in linear
parabolic equations has been studied by Kryspin and Mierczy\'nski
\cite{KryspinMierczynski2023ParabolicDifferentialEquationsBoundedDelay,
KryspinMierczynski2024SystemsParabolicEquationsDelays}.
Measures also arise as kernels of nonlocal time-delay terms in semilinear
parabolic control problems
\cite{CasasMateosTroeltzsch2018MeasureControlNonlocalTimeDelay}.

The behaviour of the delayed part of the kernel near the origin is especially
important. If its support is contained in \([\tau_*,r]\) for some
\(\tau_*>0\), then on each time interval shorter than \(\tau_*\) the delayed
term is determined entirely by the prescribed history and by the solution
already constructed on preceding intervals. The equation can then be solved
successively by the method of steps. This construction is no longer
available when the support of the delayed part accumulates at the origin,
since the memory term then samples states at times arbitrarily close to the
present.

The failure of the method of steps is not the only difficulty associated
with delayed diffusion. Even for a fixed positive delay in the principal
spatial part, well-posedness can depend on the spatial order of the
present-time regularization. Khusainov, Pokojovy, and Racke proved
well-posedness for a suitably regularized heat equation of this type and
showed that regularization of lower spatial order is, in general, insufficient
\cite{KhusainovPokojovyRacke2015StrongMildExtrapolatedL2SolutionsHeatEquationConstantDelay}.

Mass near the origin does not by itself preclude well-posedness. Ishizaka
established well-posedness and kernel stability for linear diffusion
equations with memory represented by an arbitrary finite nonnegative measure
\cite{Ishizaka2026WellPosednessKernelStabilityMixedMeasureMemory}.
In a semilinear setting with signed measure-valued delay, Shikhman obtained
uniqueness under a one-sided Lipschitz condition
\cite{Shikhman2026KernelRobustDynamicsMeasureValuedDelay}.
These results suggest that the relevant issue is not the proximity of the
kernel to the origin alone, but the structure of the combined present-time
and memory operator.

In this paper, we study nonlinear delayed diffusion with finite signed
measure kernels. If the kernel has an atom at the origin, its contribution
is instantaneous and is therefore absorbed into an effective present-time
operator. The remaining measure is treated as the residual delay kernel.
Although the residual kernel has no atom at the origin, its support may
still reach the origin. Thus, the framework covers both kernels whose
supports are separated from the origin and kernels whose supports contain
positive delays arbitrarily close to zero. The key structural assumption in
our well-posedness theory is a pathwise coercivity condition on the total
memory operator for pairs of trajectories with the same prescribed past.
This condition concerns the combined present-time and delayed contributions,
rather than requiring monotonicity of \(\mathcal C\) alone. Under this
condition, together with uniform \(p\)-monotonicity, growth, and coercivity
assumptions on the effective present-time operators and suitable growth and
local Lipschitz conditions on \(\mathcal C\), we establish existence and
uniqueness of weak solutions, together with the corresponding a priori
estimates.

We also establish stability under weak-star convergence of the total memory
kernels. It is essential here to regard each kernel as a whole, rather than
to require separate convergence of its atom at the origin and its residual
part. Indeed, an atom located at a positive delay may approach the origin
and become an instantaneous contribution in the limit. The stability
theorem captures precisely this phenomenon: the mass of such an atom is
absorbed into the effective present-time operator, and the corresponding
weak solutions converge in the natural energy space.

Finally, we verify the abstract assumptions for two classes of problems of
\(p\)-Laplacian type. In the first, the residual kernels are separated from
the origin. The second is a regularized class in which the residual kernels
may reach the origin; in this setting, the pathwise coercivity condition
follows from a smallness assumption on their total variation.

\subsection{Setting and assumptions}

Let $\Omega\subset\R^d$ be a bounded Lipschitz domain. We fix
\[
 2\le p<\infty,\qquad p'=\frac{p}{p-1}.
\]
Unless otherwise stated, we write
\[
 W^{1,p}_0=W^{1,p}_0(\Omega),\qquad L^2=L^2(\Omega),\qquad W^{-1,p'}=W^{-1,p'}(\Omega).
\]
The duality between $W^{-1,p'}$ and $W^{1,p}_0$ is denoted by $\dual{\cdot}{\cdot}$.

Let $T>0$ and $r>0$. We prescribe the initial datum and the past history by
\[
 u_0\in L^2,\qquad
 \phi\in L^p(-r,0;W^{1,p}_0).
\]
For a function $u$ on $[0,T]$, we define its history extension $\bar u$ by
\[
 \bar u(t)=
 \begin{cases}
  \phi(t)& -r<t<0,\\
  u(t)& 0\le t\le T.
 \end{cases}
\]
The value assigned at the single point $t=0$ does not affect any time
integral. The initial condition is imposed separately as $u(0)=u_0$ in $L^2$. 
In particular, no compatibility condition between $\phi$ and $u_0$
is required.

We denote by $\M([0,r])$ the space of finite signed Borel measures on $[0,r]$. 
For $\eta\in\M([0,r])$, $|\eta|$ denotes its total variation
measure, and we set
\[
 \|\eta\|_{\M([0,r])}:=|\eta|([0,r]).
\]
We write $\eta_n\stackrel{*}{\rightharpoonup}\eta$ in $\M([0,r])$ if
\[
 \int_{[0,r]}\chi\,d\eta_n\to\int_{[0,r]}\chi\,d\eta
 \qquad\text{for every }\chi\in C([0,r]).
\]
This convergence is referred to as weak-star convergence in $\M([0,r])$.

Let $\mathcal C:W^{1,p}_0\to W^{-1,p'}$ be the nonlinear spatial operator appearing in the delay term.
We start from the form
\[
 u'(t)+\mathcal A_{\rm raw}(u(t))
 +\int_{[0,r]}\mathcal C(\bar u(t-s))\,d\mu(s)
 =f(t),
 \qquad 0<t<T.
\]
Here $\mu\in\M([0,r])$ is a measure kernel. We decompose it as
\[
 \mu=\kappa\delta_0+\nu,\qquad
 \kappa=\mu(\{0\}),\qquad
 |\nu|(\{0\})=0,
\]
where $\delta_0$ denotes the Dirac measure at the origin. Then
\[
 \int_{[0,r]}\mathcal C(\bar u(t-s))\,d\mu(s)=
 \kappa\mathcal C(u(t))+\int_{(0,r]}\mathcal C(\bar u(t-s))\,d\nu(s).
\]
The first term comes from the atom at $s=0$; it is evaluated at the present
time and hence is not a delay term. We absorb it into the present-time operator by setting
\[
 \mathcal A_\kappa(u):=\mathcal A_{\rm raw}(u)+\kappa\mathcal C(u).
\]
The coefficient $\kappa$ may have either sign. With this convention,
$\nu$ is the residual delay kernel. We fix a closed interval $I\subset\R$ of
admissible values of $\kappa$ and require uniform estimates for the
corresponding effective present-time operators.

\begin{assumption}\label{ass:current}
For each $\kappa\in I$, the operator
\[
 \mathcal A_\kappa:W^{1,p}_0\to W^{-1,p'}
\]
is hemicontinuous.  There exist constants $\alpha_I>0$ and $C_I>0$ such
that, for all $\kappa\in I$ and all $u,v\in W^{1,p}_0$,
\begin{align}
 \dual{\mathcal A_\kappa(u)-\mathcal A_\kappa(v)}{u-v}
 &\ge \alpha_I\|u-v\|_{W^{1,p}_0}^p,
 \label{eq:uniform-p-monotone}\\
 \|\mathcal A_\kappa(u)\|_{W^{-1,p'}}
 &\le C_I\bigl(1+\|u\|_{W^{1,p}_0}^{p-1}\bigr),
 \label{eq:Akappa-growth}\\
 \dual{\mathcal A_\kappa(u)}{u}
 &\ge \alpha_I\|u\|_{W^{1,p}_0}^p-C_I.
 \label{eq:Akappa-coercive}
\end{align}
\end{assumption}
We shall refer to \eqref{eq:uniform-p-monotone} as uniform \(p\)-monotonicity.
This estimate will be used below together with the coercivity condition
imposed on the total memory operator.

We next impose assumptions on the flux operator appearing in the delayed term.
\begin{assumption}\label{ass:delayed-flux}
The operator
\[
 \mathcal C:W^{1,p}_0\to W^{-1,p'}
\]
is continuous.  There exists a constant $C_{\mathcal C}>0$ such that, for every
$u\in W^{1,p}_0$,
\begin{equation}
 \|\mathcal C(u)\|_{W^{-1,p'}}
 \le C_{\mathcal C}\bigl(1+\|u\|_{W^{1,p}_0}^{p-1}\bigr).
 \label{eq:C-growth}
\end{equation}
Moreover, for every bounded interval $J\subset\R$ and every $R>0$, there
exists a constant $L_{J,R}>0$ such that, whenever
\[
 \|w\|_{L^p(J;W^{1,p}_0)}
 +
 \|z\|_{L^p(J;W^{1,p}_0)}
 \le R,
\]
one has
\begin{equation}
 \|\mathcal C(w)-\mathcal C(z)\|_{L^{p'}(J;W^{-1,p'})}
 \le
 L_{J,R}\|w-z\|_{L^p(J;W^{1,p}_0)}.
 \label{eq:C-local-lip}
\end{equation}
\end{assumption}

\begin{remark}\label{rem:p-laplace-flux}
A basic example covered by Assumption \ref{ass:delayed-flux} is the
$p$-Laplacian type operator
\[
 \mathcal C(u)=-\diver(|\nabla u|^{p-2}\nabla u),
 \qquad p\ge2.
\]
For this operator, the growth estimate and continuity follow from the usual
$p$-Laplacian estimates. The local Lipschitz estimate \eqref{eq:C-local-lip} follows from the pointwise inequality
\[
 \bigl||\xi|^{p-2}\xi-|\eta|^{p-2}\eta\bigr|
 \le c_p(|\xi|+|\eta|)^{p-2}|\xi-\eta|
\]
together with Hölder's inequality in space and time. We verify this example in Section~\ref{sec:examples}.
\end{remark}

We now introduce the delay operator associated with a residual kernel.
Let $\nu\in\M([0,r])$ satisfy $|\nu|(\{0\})=0$.
For $w\in L^p(-r,T;W^{1,p}_0)$, we set
\begin{equation}
 \mathcal B_\nu w
 :=
 \int_{(0,r]}\mathcal C(w(\cdot-s))\,d\nu(s)
 \quad\text{in }L^{p'}(0,T;W^{-1,p'}).
 \label{eq:Bnu-def}
\end{equation}
Equivalently, after pairing with any fixed $v\in W^{1,p}_0$, one has, for a.e. $t\in(0,T)$,
\begin{equation}
 \dual{(\mathcal B_\nu w)(t)}{v}
 =
 \int_{(0,r]}
 \dual{\mathcal C(w(t-s))}{v}\,d\nu(s).
 \label{eq:Bnu-dual-def}
\end{equation}
The growth estimate for $\mathcal C$ gives
\begin{equation}
 \|\mathcal B_\nu w\|_{L^{p'}(0,T;W^{-1,p'})}
 \le
 C\|\nu\|_{\M([0,r])}
 \bigl(1+\|w\|_{L^p(-r,T;W^{1,p}_0)}^{p-1}\bigr),
 \label{eq:Bnu-growth}
\end{equation}
where $C$ depends only on $T$, $r$, $p$, and $C_{\mathcal C}$.

Let $\kappa\in I$ and let $\nu\in\M([0,r])$ satisfy
$|\nu|(\{0\})=0$. Given
\[
 f\in L^{p'}(0,T;W^{-1,p'}),
\]
we consider the problem
\begin{equation}
 u'(t)+\mathcal A_\kappa(u(t))+(\mathcal B_\nu\bar u)(t)=f(t),
 \qquad 0<t<T,
 \label{eq:effective-problem}
\end{equation}
together with the initial datum $u_0$ and the past history $\phi$.
\begin{definition}\label{def:weak-solution}
A function $u$ is a weak solution of \eqref{eq:effective-problem} on
$[0,T]$ with initial datum $u_0$ and past history $\phi$ if
\[
 u\in L^p(0,T;W^{1,p}_0)\cap C([0,T];L^2),
 \qquad
 u'\in L^{p'}(0,T;W^{-1,p'}),
\]
$u(0)=u_0$ in $L^2$, and
\begin{equation}
 u'+\mathcal A_\kappa(u)+\mathcal B_\nu\bar u=f
 \quad\text{in }L^{p'}(0,T;W^{-1,p'}).
 \label{eq:weak-equation}
\end{equation}
Equivalently, for every $\psi\in L^p(0,T;W^{1,p}_0)$,
\begin{align}
 &\int_0^T\dual{u'(t)}{\psi(t)}\,dt
 +\int_0^T\dual{\mathcal A_\kappa(u(t))}{\psi(t)}\,dt
 \notag\\
 +&
 \int_0^T\dual{(\mathcal B_\nu\bar u)(t)}{\psi(t)}\,dt
 =
 \int_0^T\dual{f(t)}{\psi(t)}\,dt.
 \label{eq:weak-formulation}
\end{align}
\end{definition}

We first distinguish kernels whose support stays a positive distance away
from the origin. A measure $\rho\in\M([0,r])$ is said to be separated from
zero if there exists $\tau_*>0$ such that $\supp|\rho|\subset[\tau_*,r]$.
For such kernels, the delayed argument $t-s$ is separated from the present
time $t$ by at least $\tau_*$. Thus the delay term does not sample the
solution at times arbitrarily close to the present. 

For $\nu\in\M([0,r])$ with $|\nu|(\{0\})=0$, we set
\begin{equation}
 \nu_\e:=\nu|_{[\e,r]},
 \qquad 0<\e\le r.
 \label{eq:nu-eps-def}
\end{equation}
Then $\nu_\e$ is separated from zero, and
\[
 \|\nu-\nu_\e\|_{\M([0,r])}
 \le |\nu|((0,\e])\to0
 \qquad\text{as }\e\to 0.
\]

We write
\[
\begin{aligned}
 \mathcal X_T
 &:=
 C([0,T];L^2)\cap L^p(0,T;W^{1,p}_0),\\
 \|u\|_{\mathcal X_T}
 &:=
 \|u\|_{C([0,T];L^2)}
 +\|u\|_{L^p(0,T;W^{1,p}_0)}.
\end{aligned}
\]

The main results require the following coercivity condition on the total
memory operator.

\begin{assumption}\label{ass:pathwise-coercivity}
Let $\KM\subset \M([0,r])$. For $\mu\in\KM$, set
\[
 \kappa_\mu:=\mu(\{0\}),\qquad
 \nu_\mu:=\mu-\kappa_\mu\delta_0 .
\]
We assume that $\kappa_\mu\in I$ for every $\mu\in\KM$.

For $w\in L^p(-r,T;W^{1,p}_0)$ with $w|_{[0,T]}\in\mathcal X_T$, define
\[
 (\mathcal F_\mu w)(t)
 :=
 \mathcal A_{\kappa_\mu}(w(t))
 +
 (\mathcal B_{\nu_\mu}w)(t),
 \qquad 0<t<T.
\]
There exists a constant $\gamma>0$ such that, for every
$\mu\in\KM$, every $0<t\le T$, and every
$w,z\in L^p(-r,T;W^{1,p}_0)$ satisfying
\[
 w|_{[0,T]},\,z|_{[0,T]}\in\mathcal X_T,
 \qquad
 w=z \quad\text{a.e. on }(-r,0),
\]
one has
\begin{equation}
\int_0^t
\dual{
 (\mathcal F_\mu w)(\tau)-(\mathcal F_\mu z)(\tau)
}{
 w(\tau)-z(\tau)
}
\,d\tau
\ge
\gamma
\int_0^t
\|w(\tau)-z(\tau)\|_{W^{1,p}_0}^p\,d\tau .
\label{eq:pathwise-coercivity}
\end{equation}
\end{assumption}

This condition is imposed on the total memory operator, rather than on the
delayed flux \(\mathcal C\) alone.

\subsection{Main results}

We first state the well-posedness result for general residual kernels; their
supports are not required to be separated from the origin.

\begin{theorem}
\label{thm:extension-no-atom}
Let $\kappa\in I$ and let $\nu\in\M([0,r])$ satisfy $|\nu|(\{0\})=0$.
Set
\[
 \mu:=\kappa\delta_0+\nu .
\]
Suppose that Assumptions~\ref{ass:current}, \ref{ass:delayed-flux}, and
\ref{ass:pathwise-coercivity} hold, where in
Assumption~\ref{ass:pathwise-coercivity} we take $\KM=\{\mu\}$.
Let
\[
 f\in L^{p'}(0,T;W^{-1,p'}),\qquad
 u_0\in L^2,\qquad
 \phi\in L^p(-r,0;W^{1,p}_0).
\]
Then the problem
\begin{equation}
 u'(t)+\mathcal A_\kappa(u(t))+(\mathcal B_\nu\bar u)(t)=f(t),
 \qquad 0<t<T,
 \label{eq:no-atom-problem}
\end{equation}
with initial datum \(u_0\) and past history \(\phi\) has a unique weak
solution \(u\in\mathcal X_T\). Furthermore,
\begin{align}
 &\|u\|_{C([0,T];L^2)}^2
 +\|u\|_{L^p(0,T;W^{1,p}_0)}^p
 +\|u'\|_{L^{p'}(0,T;W^{-1,p'})}^{p'}
 \notag\\
 \le{}&
 C_T\left(
 1+\|u_0\|_{L^2}^2
 +\|\phi\|_{L^p(-r,0;W^{1,p}_0)}^p
 +\|f\|_{L^{p'}(0,T;W^{-1,p'})}^{p'}
 \right).
 \label{eq:no-atom-estimate}
\end{align}
Here \(C_T\) depends only on \(T\), \(r\), \(p\), \(\alpha_I\), \(C_I\),
\(C_{\mathcal C}\), and \(\nu\).
\end{theorem}

The next theorem gives stability with respect to weak-star convergence of
total memory kernels.

\begin{theorem}
\label{thm:weakstar-stability}
Let $\mu_n,\mu\in\M([0,r])$. Set
\[
 \kappa_n:=\mu_n(\{0\}),\qquad
 \nu_n:=\mu_n-\kappa_n\delta_0,
\]
and
\[
 \kappa:=\mu(\{0\}),\qquad
 \nu:=\mu-\kappa\delta_0.
\]
Assume that
\[
 \kappa_n,\kappa\in I,
 \qquad
 \mu_n\stackrel{*}{\rightharpoonup}\mu
 \quad\text{in }\M([0,r]),
 \qquad
 \sup_n\|\mu_n\|_{\M([0,r])}<\infty .
\]
Suppose that Assumptions~\ref{ass:current}, \ref{ass:delayed-flux}, and
\ref{ass:pathwise-coercivity} hold, where in
Assumption~\ref{ass:pathwise-coercivity} we take
\[
 \KM=\{\mu\}\cup\{\mu_n:n\in\mathbb N\}.
\]
Let the data satisfy
\[
\begin{aligned}
 u_{0,n}&\to u_0
 &&\text{in }L^2,\\
 \phi_n&\to\phi
 &&\text{in }L^p(-r,0;W^{1,p}_0),\\
 f_n&\to f
 &&\text{in }L^{p'}(0,T;W^{-1,p'}).
\end{aligned}
\]
Let $u_n$ be the weak solution of
\begin{equation}
 u_n'(t)+\mathcal A_{\kappa_n}(u_n(t))
 +(\mathcal B_{\nu_n}\bar u_n)(t)=f_n(t),
 \qquad 0<t<T,
 \label{eq:weakstar-problem-n}
\end{equation}
with initial datum $u_{0,n}$ and past history $\phi_n$. Let $u$ be the weak
solution of
\begin{equation}
 u'(t)+\mathcal A_{\kappa}(u(t))
 +(\mathcal B_{\nu}\bar u)(t)=f(t),
 \qquad 0<t<T,
 \label{eq:weakstar-problem-limit}
\end{equation}
with initial datum $u_0$ and past history $\phi$. Then
\[
 u_n\to u
 \quad\text{in }\mathcal X_T.
\]
\end{theorem}

This formulation also includes the collapse of delayed atoms to the origin;
see Example~\ref{ex:collapsing-atom}.

We briefly describe the proof strategy. We begin with residual kernels whose
support is separated from the origin. After dividing the time interval into
subintervals shorter than the minimum delay, the history term on each
subinterval depends only on the prescribed past and on the solution already
constructed. The equation can therefore be solved successively by applying
the standard theory of non-delayed monotone evolution equations.

For a general residual kernel \(\nu\), we consider the truncations
\[
 \nu_\varepsilon:=\nu|_{[\varepsilon,r]}.
\]
The corresponding kernels are separated from the origin. Uniform estimates
are obtained by splitting the delay measure into a part away from zero and a
near-origin part of small total variation. We then rewrite the truncated
equations using the original total kernel. The discarded part appears as an
error term that tends to zero in the natural dual space. The pathwise
coercivity condition for the total memory operator yields a Cauchy estimate
for the approximate solutions and hence allows us to pass to the limit.

For the stability theorem, the atom at the origin and the residual part are
kept together throughout the limiting argument. A vector-valued
weak-star convergence lemma identifies the limit of the full memory terms,
including any mass transferred from positive delays to the origin. Comparing
the solutions by means of the same pathwise coercivity estimate then gives
strong convergence in \(\mathcal X_T\).

The remainder of the paper is organized as follows.  In Section~\ref{sec:prelim}
we prove the basic estimates for the delay operator and construct solutions
for kernels separated from the origin.  Section~\ref{sec:extension} is devoted
to the proof of Theorems~\ref{thm:extension-no-atom} and
\ref{thm:weakstar-stability}. In Section~\ref{sec:examples} we
discuss examples of admissible kernels and verify the assumptions for
$p$-Laplacian type fluxes.

\section{Separated kernels}\label{sec:prelim}

In this section we prove well-posedness for kernels separated from the
origin. Let $\rho\in\M([0,r])$ satisfy $\supp|\rho|\subset[\tau_*,r]$
for some $\tau_*>0$. The key point is that, on any time interval
$J=(a,b)$ with $b-a<\tau_*$, one has
\[
 t-s<a
 \qquad
 (t\in J,\ s\in\supp|\rho|).
\]
Thus the delayed term only involves values from the past of the interval.

In the estimates below, \(C\) denotes a positive constant which may change
from line to line. Numbered constants \(C_1,C_2,\dots\) are fixed in the
estimates where they are introduced, and their dependence is indicated
explicitly.

We shall use the following local Lipschitz estimate for the delay operator.

\begin{lemma}\label{lem:Bnu-lip}
Let \(\eta\in\M([0,r])\).  For every \(R>0\), set
\[
 C_1:=L_{(0,T),R},
\]
where \(L_{(0,T),R}\) is the constant in \eqref{eq:C-local-lip} with
\(J=(0,T)\).  Then, whenever
\(w,z\in L^p(-r,T;W^{1,p}_0)\) satisfy
\[
 \|w\|_{L^p(-r,T;W^{1,p}_0)}
 +
 \|z\|_{L^p(-r,T;W^{1,p}_0)}
 \le R,
\]
one has
\begin{equation}
 \|\mathcal B_\eta w-\mathcal B_\eta z\|_{L^{p'}(0,T;W^{-1,p'})}
 \le
 C_1\|\eta\|_{\M([0,r])}
 \|w-z\|_{L^p(-r,T;W^{1,p}_0)}.
 \label{eq:Bnu-lip-section2}
\end{equation}
\end{lemma}

\begin{proof}
By the definition of $\mathcal B_\eta$ and the total variation estimate for
finite signed measures,
\[
\begin{aligned}
&\|\mathcal B_\eta w-\mathcal B_\eta z\|_{L^{p'}(0,T;W^{-1,p'})} \\
\le{}&
 \int_{(0,r]}
 \bigl\|
 \mathcal C(w(\cdot-s))-\mathcal C(z(\cdot-s))
 \bigr\|_{L^{p'}(0,T;W^{-1,p'})}
 \,d|\eta|(s).
\end{aligned}
\]
By the assumption, after shifting in time, for every $s\in(0,r]$,
\[
 \|w(\cdot-s)\|_{L^p(0,T;W^{1,p}_0)}
 +
 \|z(\cdot-s)\|_{L^p(0,T;W^{1,p}_0)}
 \le R.
\]
Applying \eqref{eq:C-local-lip} with $J=(0,T)$ gives
\[
 \bigl\|
 \mathcal C(w(\cdot-s))-\mathcal C(z(\cdot-s))
 \bigr\|_{L^{p'}(0,T;W^{-1,p'})}
 \le
C_1
 \|w(\cdot-s)-z(\cdot-s)\|_{L^p(0,T;W^{1,p}_0)}.
\]
The same shift gives
\[
 \|w(\cdot-s)-z(\cdot-s)\|_{L^p(0,T;W^{1,p}_0)}
 \le
 \|w-z\|_{L^p(-r,T;W^{1,p}_0)}.
\]
Substituting these bounds into the first estimate and using
\[
 |\eta|((0,r])\le \|\eta\|_{\M([0,r])},
\]
we obtain \eqref{eq:Bnu-lip-section2}. This proves the lemma.
\end{proof}

We shall also use the following standard result for the non-delayed
present-time equation.

\begin{theorem}\label{thm:nondelayed-problem}
Suppose that Assumption \ref{ass:current} holds.
Let $J=(a,b)$ be a bounded interval and let $\kappa\in I$. For every
\[
 h\in L^{p'}(J;W^{-1,p'}),\qquad y_a\in L^2,
\]
the problem
\begin{equation}
 y'(t)+\mathcal A_\kappa(y(t))=h(t),
 \qquad t\in J,
 \qquad y(a)=y_a,
 \label{eq:nondelayed-problem}
\end{equation}
has a unique solution satisfying
\[
 y\in L^p(J;W^{1,p}_0)\cap C(\overline J;L^2),
 \qquad
 y'\in L^{p'}(J;W^{-1,p'}).
\]
Moreover,
\begin{align}
 &\|y\|_{C(\overline J;L^2)}^2
 +\|y\|_{L^p(J;W^{1,p}_0)}^p
 +\|y'\|_{L^{p'}(J;W^{-1,p'})}^{p'}
 \notag\\
 \le{}&
 C_2\left(
 1+\|y_a\|_{L^2}^2
 +\|h\|_{L^{p'}(J;W^{-1,p'})}^{p'}
 \right),
 \label{eq:nondelayed-estimate}
\end{align}
where $C_2$ depends only on $|J|$, $p$, $\alpha_I$, and $C_I$.
\end{theorem}

\begin{proof}
We first prove existence. Fix \(\kappa\in I\). By Assumption
\ref{ass:current}, the operator
\[
 \mathcal A_\kappa:W^{1,p}_0\to W^{-1,p'}
\]
is hemicontinuous. Moreover, \eqref{eq:uniform-p-monotone} implies
monotonicity, while \eqref{eq:Akappa-growth} and
\eqref{eq:Akappa-coercive} give, respectively, the \(p\)-growth and
coercivity conditions required in the standard monotone evolution theorem.
Since a monotone hemicontinuous operator is demicontinuous, see for instance
\cite[Chapter~II, Section~2]{Showalter1997MonotoneOperatorsBanachSpaceNonlinearPDE}, the demicontinuity hypothesis in
\cite[Theorem~4.10]{Barbu2010NonlinearDifferentialEquationsMonotoneTypesBanachSpaces} is satisfied.  Applying that theorem in the
setting
\[
 W^{1,p}_0\subset L^2\subset W^{-1,p'}
\]
after translating the interval \(J=(a,b)\) to \((0,b-a)\), we obtain a unique
solution of \eqref{eq:nondelayed-problem} such that
\[
y\in C(\overline J;L^2) \cap L^p(J;W^{1,p}_0) \cap W^{1,p'}(J;W^{-1,p'}).
\]

We now derive the estimates used later. Taking the duality pairing of
\eqref{eq:nondelayed-problem} with \(y(\tau)\) and integrating over
\((a,t)\), we obtain
\[
 \int_a^t\dual{y'(\tau)}{y(\tau)}\,d\tau
 +\int_a^t\dual{\mathcal A_\kappa(y(\tau))}{y(\tau)}\,d\tau
 =
 \int_a^t\dual{h(\tau)}{y(\tau)}\,d\tau .
\]
Using the energy identity,
\[
 \int_a^t\dual{y'(\tau)}{y(\tau)}\,d\tau
 =
 \frac12\int_a^t
 \frac{d}{d\tau}\|y(\tau)\|_{L^2}^2\,d\tau
 =
 \frac12\|y(t)\|_{L^2}^2
 -\frac12\|y_a\|_{L^2}^2 .
\]
By \eqref{eq:Akappa-coercive}, we get
\[
 \frac12\|y(t)\|_{L^2}^2
 +\alpha_I\int_a^t\|y(\tau)\|_{W^{1,p}_0}^p\,d\tau
 \le
 \frac12\|y_a\|_{L^2}^2
 +C_I|J|
 +\int_a^t\dual{h(\tau)}{y(\tau)}\,d\tau .
\]
Young's inequality gives
\[
 \dual{h(\tau)}{y(\tau)}
 \le
 \frac{\alpha_I}{2}\|y(\tau)\|_{W^{1,p}_0}^p
 +C\|h(\tau)\|_{W^{-1,p'}}^{p'} .
\]
Therefore,
\[
 \frac12\|y(t)\|_{L^2}^2
 +\frac{\alpha_I}{2}
 \int_a^t\|y(\tau)\|_{W^{1,p}_0}^p\,d\tau
 \le
 \frac12\|y_a\|_{L^2}^2
 +C_I|J|
 +C\int_a^t\|h(\tau)\|_{W^{-1,p'}}^{p'}\,d\tau .
\]
Taking the supremum over \(t\in\overline J\), we obtain
\[
 \|y\|_{C(\overline J;L^2)}^2
 +\|y\|_{L^p(J;W^{1,p}_0)}^p
 \le
 C \left(
 1+\|y_a\|_{L^2}^2
 +\|h\|_{L^{p'}(J;W^{-1,p'})}^{p'}
 \right).
\]
Moreover,
\[
 y'=h-\mathcal A_\kappa(y).
\]
The growth estimate \eqref{eq:Akappa-growth} then yields
\[
 \|y'\|_{L^{p'}(J;W^{-1,p'})}^{p'}
 \le
 C\left(
 1+\|y\|_{L^p(J;W^{1,p}_0)}^p
 +\|h\|_{L^{p'}(J;W^{-1,p'})}^{p'}
 \right).
\]
Together with the previous bound, this proves
\eqref{eq:nondelayed-estimate} and completes the proof.
\end{proof}

The following lemma gives the corresponding difference estimate and
stability property for the non-delayed problem.

\begin{lemma}\label{lem:nondelayed-stability}
Suppose that Assumptions \ref{ass:current} and
\ref{ass:delayed-flux} hold.  Let \(J=(a,b)\) be a bounded interval and
let \(\kappa\in I\).  Let \(y_1,y_2\) be the solutions of
\eqref{eq:nondelayed-problem} corresponding to the same \(\kappa\) and
to data
\[
 h_i\in L^{p'}(J;W^{-1,p'}),\qquad y_{a,i}\in L^2,
 \qquad i=1,2.
\]
Then
\begin{align}
 &\|y_1-y_2\|_{C(\overline J;L^2)}^2
 +\|y_1-y_2\|_{L^p(J;W^{1,p}_0)}^p
 \notag\\
 \le{}&
 C_3\left(
 \|y_{a,1}-y_{a,2}\|_{L^2}^2
 +\|h_1-h_2\|_{L^{p'}(J;W^{-1,p'})}^{p'}
 \right),
 \label{eq:nondelayed-difference}
\end{align}
where \(C_3\) depends only on \(p\) and \(\alpha_I\).

Moreover, the solution is stable with respect to \(\kappa\) and the data.
Let \(\kappa_n\to\kappa\) in \(I\),
\[
 h_n\to h
  \quad\text{in }L^{p'}(J;W^{-1,p'}),
 \qquad
 y_{a,n}\to y_a
 \quad\text{in }L^2.
\]
Let \(y_n\) and \(y\) be the corresponding solutions for
\((\kappa_n,h_n,y_{a,n})\) and \((\kappa,h,y_a)\).  Then
\[
 y_n\to y
 \quad\text{in }
 C(\overline J;L^2)\cap L^p(J;W^{1,p}_0).
\]
\end{lemma}

\begin{proof}
We first derive the difference estimate.  Set $ \zeta:=y_1-y_2 $.
Then
\[
 \zeta'
 +\mathcal A_\kappa(y_1)-\mathcal A_\kappa(y_2)
 =
 h_1-h_2 .
\]
Taking the duality pairing with \(\zeta(\tau)\) and integrating over
\((a,t)\), we obtain
\[
\begin{aligned}
 &\int_a^t\dual{\zeta'(\tau)}{\zeta(\tau)}\,d\tau
 +\int_a^t
 \dual{\mathcal A_\kappa(y_1(\tau))-\mathcal A_\kappa(y_2(\tau))}
       {\zeta(\tau)}
 \,d\tau \\
 &\qquad=
 \int_a^t\dual{h_1(\tau)-h_2(\tau)}{\zeta(\tau)}\,d\tau .
\end{aligned}
\]
Using the energy identity, we have
\[
 \int_a^t\dual{\zeta'(\tau)}{\zeta(\tau)}\,d\tau
 =
 \frac12\|\zeta(t)\|_{L^2}^2
 -\frac12\|y_{a,1}-y_{a,2}\|_{L^2}^2 .
\]
By the uniform \(p\)-monotonicity \eqref{eq:uniform-p-monotone},
\[
 \dual{\mathcal A_\kappa(y_1(\tau))-\mathcal A_\kappa(y_2(\tau))}
       {\zeta(\tau)}
 \ge
 \alpha_I\|\zeta(\tau)\|_{W^{1,p}_0}^p .
\]
Hence
\[
\begin{aligned}
 &\frac12\|\zeta(t)\|_{L^2}^2
 +\alpha_I\int_a^t\|\zeta(\tau)\|_{W^{1,p}_0}^p\,d\tau \\
\le{}&
 \frac12\|y_{a,1}-y_{a,2}\|_{L^2}^2 
 +\int_a^t\dual{h_1(\tau)-h_2(\tau)}{\zeta(\tau)}\,d\tau .
\end{aligned}
\]
Young's inequality gives
\[
 \dual{h_1(\tau)-h_2(\tau)}{\zeta(\tau)}
 \le
 \frac{\alpha_I}{2}\|\zeta(\tau)\|_{W^{1,p}_0}^p
 +C\|h_1(\tau)-h_2(\tau)\|_{W^{-1,p'}}^{p'} .
\]
Therefore,
\[
\begin{aligned}
 &\frac12\|\zeta(t)\|_{L^2}^2
 +\frac{\alpha_I}{2}
 \int_a^t\|\zeta(\tau)\|_{W^{1,p}_0}^p\,d\tau \\
 \le{} &
 \frac12\|y_{a,1}-y_{a,2}\|_{L^2}^2
 +C\int_a^t
 \|h_1(\tau)-h_2(\tau)\|_{W^{-1,p'}}^{p'}\,d\tau .
\end{aligned}
\]
Taking the supremum over \(t\in\overline J\), and taking \(t=b\) in the
integral term, we obtain \eqref{eq:nondelayed-difference}.

It remains to prove stability with respect to \(\kappa\). Let \(y_n\) and
\(y\) be the solutions corresponding to \((\kappa_n,h_n,y_{a,n})\) and
\((\kappa,h,y_a)\), respectively. Set \(\zeta_n:=y_n-y\) and define
\[
 \widetilde h_n
 :=
 h+\mathcal A_{\kappa_n}(y)-\mathcal A_\kappa(y).
\]
By the definition
\[
 \mathcal A_\kappa=\mathcal A_{\rm raw}+\kappa\mathcal C,
\]
we have
\[
 \mathcal A_{\kappa_n}(y)-\mathcal A_\kappa(y)
 =
 (\kappa_n-\kappa)\mathcal C(y).
\]
Since \(y\in L^p(J;W^{1,p}_0)\), the growth estimate
\eqref{eq:C-growth} gives
\[
 \mathcal C(y)\in L^{p'}(J;W^{-1,p'}).
\]
Thus \(\widetilde h_n\in L^{p'}(J;W^{-1,p'})\), and
\[
 y'
 +\mathcal A_{\kappa_n}(y)
 =
 \widetilde h_n,
 \qquad
 y(a)=y_a.
\]
Applying \eqref{eq:nondelayed-difference}, with \(\kappa\) replaced by
\(\kappa_n\), gives
\begin{align*}
 &\|\zeta_n\|_{C(\overline J;L^2)}^2
 +\|\zeta_n\|_{L^p(J;W^{1,p}_0)}^p\\
 \le{}&
 C_3 \left(
 \|y_{a,n}-y_a\|_{L^2}^2
 +\|h_n-\widetilde h_n\|_{L^{p'}(J;W^{-1,p'})}^{p'}
 \right)\\
 \le{}&
 C\Bigl(
 \|y_{a,n}-y_a\|_{L^2}^2
 +\|h_n-h\|_{L^{p'}(J;W^{-1,p'})}^{p'}\\
 &
 +\|\mathcal A_{\kappa_n}(y)-\mathcal A_\kappa(y)\|_{L^{p'}(J;W^{-1,p'})}^{p'}
 \Bigr).
\end{align*}
Moreover,
\[
 \mathcal A_{\kappa_n}(y)\to\mathcal A_\kappa(y)
 \quad\text{in }L^{p'}(J;W^{-1,p'})
\]
because \(\kappa_n\to\kappa\) and \(\mathcal C(y)\in L^{p'}(J;W^{-1,p'})\).
Together with
\[
 y_{a,n}\to y_a \quad\text{in }L^2,
 \qquad
 h_n\to h \quad\text{in }L^{p'}(J;W^{-1,p'}),
\]
\eqref{eq:nondelayed-difference} implies
\[
 y_n\to y
 \quad\text{in }
 C(\overline J;L^2)\cap L^p(J;W^{1,p}_0).
\]
This completes the proof.
\end{proof}

We next prove existence, uniqueness, and an a priori estimate for kernels
separated from the origin.

\begin{theorem}\label{thm:separated-wp}
Suppose that Assumptions \ref{ass:current} and \ref{ass:delayed-flux} hold.
Let \(\kappa\in I\) and let \(\rho\in\M([0,r])\) satisfy
\[
 \supp|\rho|\subset[\tau_*,r]
\]
for some \(0<\tau_*\le r\).  Then, for every
\[
 f\in L^{p'}(0,T;W^{-1,p'}),\qquad
 u_0\in L^2,\qquad
 \phi\in L^p(-r,0;W^{1,p}_0),
\]
the problem
\begin{equation}
 u'(t)+\mathcal A_\kappa(u(t))+(\mathcal B_\rho\bar u)(t)=f(t),
 \qquad 0<t<T,
 \label{eq:separated-problem}
\end{equation}
has a unique weak solution with initial datum \(u_0\) and past history
\(\phi\).  Moreover,
\begin{align}
 &\|u\|_{C([0,T];L^2)}^2
 +\|u\|_{L^p(0,T;W^{1,p}_0)}^p
 +\|u'\|_{L^{p'}(0,T;W^{-1,p'})}^{p'}
 \notag\\
 \le{}&
 C_4\left(
 1+\|u_0\|_{L^2}^2
 +\|\phi\|_{L^p(-r,0;W^{1,p}_0)}^p
 +\|f\|_{L^{p'}(0,T;W^{-1,p'})}^{p'}
 \right),
 \label{eq:separated-estimate}
\end{align}
where \(C_4\) depends only on \(T\), \(r\), \(\tau_*\), \(p\),
\(\alpha_I\), \(C_I\), \(C_{\mathcal C}\), and
\(\|\rho\|_{\M([0,r])}\).
\end{theorem}
\begin{proof}
Choose \(N\in\mathbb N\) such that \(T/N<\tau_*\), and set
\[
 t_j:=\frac{jT}{N},\qquad j=0,\dots,N.
\]
Then
\[
 0=t_0<t_1<\cdots<t_N=T,
 \qquad
 t_j-t_{j-1}<\tau_* .
\]
We construct the solution successively on the intervals \(I_j=(t_{j-1},t_j)\).

For the induction step, suppose that \(u\) has already been constructed on
the previous intervals \((0,t_{j-1})\), with the convention that this set is
empty when \(j=1\). Define \(w_j\in L^p(-r,T;W^{1,p}_0)\) by
\[
 w_j(t)=
 \begin{cases}
  \phi(t) & -r<t<0,\\
  u(t)   & 0\le t \le t_{j-1},\\
  0      & t_{j-1}<t<T.
 \end{cases}
\]
If \(t\in I_j\) and
\(s\in\supp|\rho|\), then \(s\ge\tau_*\), and hence
\[
 t-s<t_j-\tau_*<t_{j-1}.
\]
Thus the delayed term on \(I_j\) only involves already known values.  Set
\[
 F_j(t):=(\mathcal B_\rho w_j)(t),
 \qquad t\in I_j.
\]
By \eqref{eq:Bnu-growth}, applied with \(\nu=\rho\),
\[
 F_j\in L^{p'}(I_j;W^{-1,p'}).
\]
We now solve
\begin{equation}
 u'(t)+\mathcal A_\kappa(u(t))=f(t)-F_j(t),
 \qquad t\in I_j,
 \label{eq:step-problem}
\end{equation}
with initial value \(u_0\) if \(j=1\), and with initial value
\(u(t_{j-1})\) if \(j\ge2\). Theorem
\ref{thm:nondelayed-problem} gives a unique solution on \(I_j\), which we
append to the previously constructed one.
Repeating the construction for \(j=1,\dots,N\), we obtain
\[
 u\in L^p(0,T;W^{1,p}_0)\cap C([0,T];L^2),
 \qquad
 u'\in L^{p'}(0,T;W^{-1,p'}).
\]
By construction, \(F_j=(\mathcal B_\rho\bar u)|_{I_j}\) on each interval
\(I_j\), and hence \(u\) is a weak solution of
\eqref{eq:separated-problem}.

Uniqueness is proved in the same stepwise way. On \(I_1\), the delayed
term is determined entirely by the prescribed history \(\phi\). Thus two
solutions with the same data solve the same non-delayed problem on \(I_1\),
and they coincide there. If they coincide on \((-r,t_{j-1}]\), then their
separated delay terms coincide on \(I_j\). The uniqueness part of Theorem
\ref{thm:nondelayed-problem} implies that the two solutions coincide on
\(I_j\). Induction over \(j\) proves uniqueness on \([0,T]\).

It remains to prove the estimate.  Put
\[
 E_j:=
 \|u\|_{C([0,t_j];L^2)}^2
 +
 \|u\|_{L^p(0,t_j;W^{1,p}_0)}^p.
\]
By the definition of \(F_j\) and the growth estimate
\eqref{eq:Bnu-growth}, we have
\[
 \|F_j\|_{L^{p'}(I_j;W^{-1,p'})}^{p'}
 \le
 C\left(
 1+\|\phi\|_{L^p(-r,0;W^{1,p}_0)}^p
  +\|u\|_{L^p(0,t_{j-1};W^{1,p}_0)}^p
 \right),
\]
where \(C\) depends only on \(T\), \(r\), \(p\), \(C_{\mathcal C}\), and
\(\|\rho\|_{\M([0,r])}\). Applying
\eqref{eq:nondelayed-estimate} on \(I_j\) to
\eqref{eq:step-problem}, we obtain
\[
 E_j
 \le
 C E_{j-1}
 +C\left(
 1+\|\phi\|_{L^p(-r,0;W^{1,p}_0)}^p
  +\|f\|_{L^{p'}(I_j;W^{-1,p'})}^{p'}
 \right),
\]
with \(E_0=\|u_0\|_{L^2}^2\). Iterating this estimate for \(j=1,\dots,N\), we obtain
\[
\begin{aligned}
 &\|u\|_{C([0,T];L^2)}^2
 +\|u\|_{L^p(0,T;W^{1,p}_0)}^p \\
\le{}&
 C\left(
 1+\|u_0\|_{L^2}^2
 +\|\phi\|_{L^p(-r,0;W^{1,p}_0)}^p
 +\|f\|_{L^{p'}(0,T;W^{-1,p'})}^{p'}
 \right).
\end{aligned}
\]
Finally,
\[
 u'=f-\mathcal A_\kappa(u)-\mathcal B_\rho\bar u.
\]
Using \eqref{eq:Akappa-growth} and \eqref{eq:Bnu-growth}, we get
\[
\begin{aligned}
 &\|u'\|_{L^{p'}(0,T;W^{-1,p'})}^{p'} \\
 \le{}&
 C\left(
 1+\|u_0\|_{L^2}^2
 +\|\phi\|_{L^p(-r,0;W^{1,p}_0)}^p
 +\|f\|_{L^{p'}(0,T;W^{-1,p'})}^{p'}
 \right).
\end{aligned}
\]
This proves \eqref{eq:separated-estimate} and completes the proof.
\end{proof}

\section{Extension and stability}\label{sec:extension}

In this section we remove the separation assumption on the residual kernel.
We first obtain uniform estimates for the separated kernels \(\nu_\e\)
defined in \eqref{eq:nu-eps-def}. For the convergence step we use the total kernel
\[
 \mu=\kappa\delta_0+\nu .
\]
The equation with \(\nu_\e\) is rewritten as the full-kernel equation plus an
error term which tends to zero in \(L^{p'}(0,T;W^{-1,p'})\). The comparison
estimate is then obtained from Assumption~\ref{ass:pathwise-coercivity}.

For \(\eta\in\M([0,r])\) and \(0<\delta\le r\), we write
\[
 m_\eta(\delta):=|\eta|((0,\delta]).
\]
Since \(|\eta|\) is a finite measure, we have
\[
 m_\eta(\delta)\to0
 \quad\text{as }\delta\to0 .
\]

We first derive an a priori estimate for the solutions corresponding to the
separated approximating kernels \(\nu_\e\), uniformly in \(\e\).

\begin{lemma}\label{lem:uniform-truncation-estimate}
Suppose that Assumptions \ref{ass:current} and \ref{ass:delayed-flux} hold.
Let \(\kappa\in I\), and let \(\nu\in\M([0,r])\) satisfy $|\nu|(\{0\})=0$.
Let
\[
 f\in L^{p'}(0,T;W^{-1,p'}),\qquad
 u_0\in L^2,\qquad
 \phi\in L^p(-r,0;W^{1,p}_0).
\]
For \(0<\e\le r\), set $ \nu_\e:=\nu|_{[\e,r]}$,
and let \(u_\e\) be the weak solution with kernel
\(\nu_\e\) and data \((u_0,\phi,f)\).  Then there exists a constant
\(C_5>0\) such that
\begin{align}
 &\|u_\e\|_{C([0,T];L^2)}^2
 +\|u_\e\|_{L^p(0,T;W^{1,p}_0)}^p
 +\|u_\e'\|_{L^{p'}(0,T;W^{-1,p'})}^{p'}
 \notag\\
 \le{}&
 C_5\left(
 1+\|u_0\|_{L^2}^2
 +\|\phi\|_{L^p(-r,0;W^{1,p}_0)}^p
 +\|f\|_{L^{p'}(0,T;W^{-1,p'})}^{p'}
 \right).
 \label{eq:uniform-truncation-estimate}
\end{align}
Here \(C_5\) may depend on \(T\), \(r\), \(p\), \(\alpha_I\), \(C_I\),
\(C_{\mathcal C}\), and \(\nu\), but not on \(\e\).
\end{lemma}

\begin{proof}
The existence of \(u_\e\) follows from Theorem
\ref{thm:separated-wp}, because \(\nu_\e\) is separated from zero.
We prove the estimate.

Let \(0<\delta\le r\) be a parameter to be chosen sufficiently small below.
For this \(\delta\), choose \(N\in\mathbb N\) such that \(T/N<\delta\), and set
\[
 t_j:=\frac{jT}{N},\qquad j=0,\dots,N.
\]
Then
\[
 0=t_0<t_1<\cdots<t_N=T,
 \qquad
 t_j-t_{j-1}<\delta .
\]
We write \(I_j=(t_{j-1},t_j)\).

For fixed \(\e\), decompose
\[
 \nu_\e
 =
 \rho_{\e,\delta}
 +
 \mu_{\e,\delta},
 \qquad
 \rho_{\e,\delta}:=\nu_\e|_{[\delta,r]},
 \qquad
 \mu_{\e,\delta}:=\nu_\e|_{(0,\delta)} .
\]
The measure \(\rho_{\e,\delta}\) is separated from zero by \(\delta\), whereas
\[
 \|\mu_{\e,\delta}\|_{\M([0,r])}
 \le m_\nu(\delta).
\]

Put
\[
 E_{\e,j}:=
 \|u_\e\|_{C([0,t_j];L^2)}^2
 +
 \|u_\e\|_{L^p(0,t_j;W^{1,p}_0)}^p .
\]
Then \(E_{\e,0}=\|u_0\|_{L^2}^2\).  On \(I_j\), the equation for
\(u_\e\) can be written as
\[
 u_\e'(t)+\mathcal A_\kappa(u_\e(t))
 =
 f(t)
 -(\mathcal B_{\rho_{\e,\delta}}\bar u_\e)(t)
 -(\mathcal B_{\mu_{\e,\delta}}\bar u_\e)(t).
\]
If \(t\in I_j\) and \(s\in\supp|\rho_{\e,\delta}|\), then
\(s\ge\delta\), and hence
\[
 t-s<t_j-\delta<t_{j-1}.
\]
Thus the term \(\mathcal B_{\rho_{\e,\delta}}\bar u_\e\) on \(I_j\) only
involves the history and the solution on \((0,t_{j-1})\). By the same
argument as in \eqref{eq:Bnu-growth}, using the fact that
\(\rho_{\e,\delta}\) is supported in \([\delta,r]\), we have
\[
 \|\mathcal B_{\rho_{\e,\delta}}\bar u_\e
 \|_{L^{p'}(I_j;W^{-1,p'})}^{p'}
 \le
 C\left(
 1+\|\phi\|_{L^p(-r,0;W^{1,p}_0)}^p
 +E_{\e,j-1}
 \right).
\]
For the near-origin part, again by the same argument as in
\eqref{eq:Bnu-growth}, we obtain
\[
 \|\mathcal B_{\mu_{\e,\delta}}\bar u_\e
 \|_{L^{p'}(I_j;W^{-1,p'})}^{p'}
 \le
 C m_\nu(\delta)^{p'}
 \left(
 1+\|\phi\|_{L^p(-r,0;W^{1,p}_0)}^p
 +E_{\e,j}
 \right).
\]
In what follows, \(C\) is independent of \(\e\) and \(j\), and the constant
in \eqref{eq:nondelayed-estimate} is taken uniformly for \(|I_j|\le T\).

Applying \eqref{eq:nondelayed-estimate} on \(I_j\), and using the preceding
two estimates, we get
\[
\begin{aligned}
 E_{\e,j}
 \le{}&
 C E_{\e,j-1}
 +C\left(
 1+\|\phi\|_{L^p(-r,0;W^{1,p}_0)}^p
 +\|f\|_{L^{p'}(I_j;W^{-1,p'})}^{p'}
 \right) \\
 &+
 K_1 m_\nu(\delta)^{p'} E_{\e,j},
\end{aligned}
\]
where \(K_1>0\) is independent of \(\e\), \(j\), and \(\delta\). Since
\(m_\nu(\delta)\to0\) as \(\delta\to0\), we now choose the parameter
\(\delta\in(0,r]\) so small that
\[
 K_1m_\nu(\delta)^{p'}\le \frac12 .
\]
For this fixed \(\delta\), the last term can be absorbed into the left-hand
side. Hence
\[
 E_{\e,j}
 \le
 C E_{\e,j-1}
 +C\left(
 1+\|\phi\|_{L^p(-r,0;W^{1,p}_0)}^p
 +\|f\|_{L^{p'}(I_j;W^{-1,p'})}^{p'}
 \right).
\]
Iterating this estimate for \(j=1,\dots,N\), we obtain
\[
\begin{aligned}
 &\|u_\e\|_{C([0,T];L^2)}^2
 +\|u_\e\|_{L^p(0,T;W^{1,p}_0)}^p \\
 \le{}&
 C\left(
 1+\|u_0\|_{L^2}^2
 +\|\phi\|_{L^p(-r,0;W^{1,p}_0)}^p
 +\|f\|_{L^{p'}(0,T;W^{-1,p'})}^{p'}
 \right),
\end{aligned}
\]
where \(C\) is independent of \(\e\).

It remains to estimate the time derivative.  From the equation,
\[
 u_\e'
 =
 f-\mathcal A_\kappa(u_\e)
  -\mathcal B_{\nu_\e}\bar u_\e .
\]
Using \eqref{eq:Akappa-growth} and \eqref{eq:Bnu-growth}, together with
\(\|\nu_\e\|_{\M([0,r])}\le\|\nu\|_{\M([0,r])}\), we get
\[
\begin{aligned}
 \|u_\e'\|_{L^{p'}(0,T;W^{-1,p'})}^{p'}
 \le{}&
 C\|f\|_{L^{p'}(0,T;W^{-1,p'})}^{p'}
 +C\left(
  1+\|u_\e\|_{L^p(0,T;W^{1,p}_0)}^p
 \right) \\
 &+
 C\left(
  1+\|\bar u_\e\|_{L^p(-r,T;W^{1,p}_0)}^p
 \right).
\end{aligned}
\]
Since
\[
 \|\bar u_\e\|_{L^p(-r,T;W^{1,p}_0)}^p
 =
 \|\phi\|_{L^p(-r,0;W^{1,p}_0)}^p
 +
 \|u_\e\|_{L^p(0,T;W^{1,p}_0)}^p,
\]
the previous bound for \(u_\e\) yields
\[
\begin{aligned}
 &\|u_\e'\|_{L^{p'}(0,T;W^{-1,p'})}^{p'}\\
 \le{}&
 C\left(
 1+\|u_0\|_{L^2}^2
 +\|\phi\|_{L^p(-r,0;W^{1,p}_0)}^p  
 +\|f\|_{L^{p'}(0,T;W^{-1,p'})}^{p'}
 \right).
\end{aligned}
\]
The last two estimates, with constants independent of \(\e\), give
\eqref{eq:uniform-truncation-estimate} for some \(C_5\).  This completes the
proof.
\end{proof}

We next show that the error produced by removing the near-origin part of the
residual measure vanishes in the natural dual norm.

\begin{lemma}\label{lem:small-kernel-error}
Suppose that Assumptions \ref{ass:current} and \ref{ass:delayed-flux} hold.
Let \(\kappa\in I\), and let \(\nu\in\M([0,r])\) satisfy
\(|\nu|(\{0\})=0\). Let
\[
 f\in L^{p'}(0,T;W^{-1,p'}),\qquad
 u_0\in L^2,\qquad
 \phi\in L^p(-r,0;W^{1,p}_0).
\]
For \(0<\e\le r\), let \(\nu_\e\) be defined by
\eqref{eq:nu-eps-def}, and let \(u_\e\) be the weak solution with kernel
\(\nu_\e\) and data \((u_0,\phi,f)\). Set
\[
 \mu:=\kappa\delta_0+\nu,
 \qquad
 E_\e:=\mathcal B_{\nu-\nu_\e}\bar u_\e .
\]
Then
\begin{equation}
 u_\e'
 +\mathcal F_\mu\bar u_\e
 =
 f+E_\e
 \quad\text{in }L^{p'}(0,T;W^{-1,p'}),
 \label{eq:full-kernel-with-small-error}
\end{equation}
and
\begin{equation}
 E_\e\to0
 \quad\text{in }L^{p'}(0,T;W^{-1,p'})
 \qquad\text{as }\e\to0 .
 \label{eq:small-kernel-error-to-zero}
\end{equation}
\end{lemma}

\begin{proof}
By \eqref{eq:nu-eps-def}, \(\nu-\nu_\e=\nu|_{[0,\e)}\)
as measures on \([0,r]\). Therefore, since \(|\nu|(\{0\})=0\),
\[
 \|\nu-\nu_\e\|_{\M([0,r])}
 =
 |\nu|([0,\e))
 =
 |\nu|((0,\e))
 \le
 m_\nu(\e)
 \to0
 \qquad\text{as }\e\to0 .
\]

The equation solved by \(u_\e\) is
\[
 u_\e'
 +\mathcal A_\kappa(u_\e)
 +\mathcal B_{\nu_\e}\bar u_\e
 =
 f
 \quad\text{in }L^{p'}(0,T;W^{-1,p'}).
\]
Since \(\mu=\kappa\delta_0+\nu\), the definition of
\(\mathcal F_\mu\) gives
\[
 \mathcal F_\mu\bar u_\e
 =
 \mathcal A_\kappa(u_\e)
 +
 \mathcal B_\nu\bar u_\e .
\]
Consequently,
\[
\begin{aligned}
 u_\e'
 +\mathcal F_\mu\bar u_\e
 &=
 u_\e'
 +\mathcal A_\kappa(u_\e)
 +\mathcal B_\nu\bar u_\e  \\
 &=
 f
 -\mathcal B_{\nu_\e}\bar u_\e
 +\mathcal B_\nu\bar u_\e  \\
 &=
 f
 +\mathcal B_{\nu-\nu_\e}\bar u_\e
 =
 f+E_\e .
\end{aligned}
\]
This proves \eqref{eq:full-kernel-with-small-error}.

It remains to prove \eqref{eq:small-kernel-error-to-zero}. By
\eqref{eq:Bnu-growth}, applied to the measure \(\nu-\nu_\e\), we have
\[
 \|E_\e\|_{L^{p'}(0,T;W^{-1,p'})}
 \le
 C\|\nu-\nu_\e\|_{\M([0,r])}
 \left(
  1+
  \|\bar u_\e\|_{L^p(-r,T;W^{1,p}_0)}^{p-1}
 \right).
\]
Moreover,
\[
 \|\bar u_\e\|_{L^p(-r,T;W^{1,p}_0)}^p
 =
 \|\phi\|_{L^p(-r,0;W^{1,p}_0)}^p
 +
 \|u_\e\|_{L^p(0,T;W^{1,p}_0)}^p .
\]
The uniform estimate \eqref{eq:uniform-truncation-estimate} implies that
\[
 \sup_{0<\e\le r}
 \|\bar u_\e\|_{L^p(-r,T;W^{1,p}_0)}
 <\infty .
\]
Hence
\[
 \|E_\e\|_{L^{p'}(0,T;W^{-1,p'})}
 \le
 C m_\nu(\e)
 \to0
 \qquad\text{as }\e\to0 .
\]
This proves \eqref{eq:small-kernel-error-to-zero}. 
\end{proof}

The next estimate compares two separated problems after both have been
written with the same total kernel. This is the step where the coercivity of
\(\mathcal F_\mu\) is used.

\begin{lemma}\label{lem:separated-approx-cauchy}
Suppose that the hypotheses of Theorem~\ref{thm:extension-no-atom} hold.
For \(0<\e\le r\), set \(\nu_\e:=\nu|_{[\e,r]}\), and let \(u_\e\) be the
weak solution with kernel \(\nu_\e\). Let \(E_\e\) be the error term defined
in Lemma~\ref{lem:small-kernel-error}. Then there exists a constant \(C_6>0\)
such that, for all \(0<\e,\sigma\le r\),
\begin{equation}
 \|u_\e-u_\sigma\|_{C([0,T];L^2)}^2
 +
 \|u_\e-u_\sigma\|_{L^p(0,T;W^{1,p}_0)}^p
 \le
 C_6
 \|E_\e-E_\sigma\|_{L^{p'}(0,T;W^{-1,p'})}^{p'} .
 \label{eq:separated-approx-cauchy}
\end{equation}
Here \(C_6\) depends only on \(p\) and \(\gamma\). In particular,
\(\{u_\e\}_{0<\e\le r}\) is a Cauchy family in \(\mathcal X_T\).
\end{lemma}

\begin{proof}
Set \( z_{\e,\sigma}:=u_\e-u_\sigma\).
By Lemma~\ref{lem:small-kernel-error},
\[
 u_\e'+\mathcal F_\mu\bar u_\e=f+E_\e,
 \qquad
 u_\sigma'+\mathcal F_\mu\bar u_\sigma=f+E_\sigma .
\]
Subtracting the two identities gives
\[
 z_{\e,\sigma}'
 +
 \mathcal F_\mu\bar u_\e
 -
 \mathcal F_\mu\bar u_\sigma
 =
 E_\e-E_\sigma .
\]
Since \(\bar u_\e\) and \(\bar u_\sigma\) have the same past history
\(\phi\), Assumption~\ref{ass:pathwise-coercivity} yields, after testing by
\(z_{\e,\sigma}\) and integrating over \((0,t)\),
\[
 \frac12\|z_{\e,\sigma}(t)\|_{L^2}^2
 +
 \gamma\int_0^t
 \|z_{\e,\sigma}(\tau)\|_{W^{1,p}_0}^p\,d\tau
 \le
 \int_0^t
 \dual{E_\e(\tau)-E_\sigma(\tau)}
      {z_{\e,\sigma}(\tau)}
 \,d\tau .
\]
Young's inequality gives
\[
 \int_0^t
 \dual{E_\e-E_\sigma}{z_{\e,\sigma}}\,d\tau
 \le
 \frac{\gamma}{2}
 \int_0^t
 \|z_{\e,\sigma}(\tau)\|_{W^{1,p}_0}^p\,d\tau
 +
 C\|E_\e-E_\sigma\|_{L^{p'}(0,T;W^{-1,p'})}^{p'} .
\]
Absorbing the first term and taking the supremum in \(t\), we obtain
\eqref{eq:separated-approx-cauchy}. Finally,
\eqref{eq:small-kernel-error-to-zero} implies
\[
 E_\e-E_\sigma\to0
 \quad\text{in }L^{p'}(0,T;W^{-1,p'})
 \qquad\text{as }\e,\sigma\to0 .
\]
Thus \(\{u_\e\}_{0<\e\le r}\) is Cauchy in \(\mathcal X_T\). The lemma follows.
\end{proof}

We are now ready to prove the extension theorem.

\begin{proof}[Proof of Theorem~\ref{thm:extension-no-atom}]
For each \(0<\e\le r\), the measure \(\nu_\e\) is separated from zero.
Hence Theorem~\ref{thm:separated-wp} gives a unique weak solution \(u_\e\)
with kernel \(\nu_\e\) and data \((u_0,\phi,f)\).

By Lemma~\ref{lem:separated-approx-cauchy}, \(\{u_\e\}_{0<\e\le r}\) is
Cauchy in \(\mathcal X_T\). Hence there exists \(u\in\mathcal X_T\) such that
\begin{equation}
 u_\e\to u
 \quad\text{in }\mathcal X_T
 \qquad\text{as }\e\to0 .
 \label{eq:ueps-to-u-XT}
\end{equation}
Then \(u(0)=u_0\) in \(L^2\). Since the past history is fixed, we also have
\begin{equation}
 \bar u_\e\to\bar u
 \quad\text{in }L^p(-r,T;W^{1,p}_0).
 \label{eq:uepsbar-to-ubar}
\end{equation}

We next pass to the limit in the equation. By
Lemma~\ref{lem:uniform-truncation-estimate}, the family \(\{u_\e'\}\) is
bounded in \(L^{p'}(0,T;W^{-1,p'})\). Combining this bound with
\eqref{eq:ueps-to-u-XT}, and using the uniqueness of the weak derivative, we
get
\[
 u_\e'\rightharpoonup u'
 \quad\text{in }L^{p'}(0,T;W^{-1,p'})
 \qquad\text{as }\e\to0 .
\]
Since a monotone hemicontinuous operator is demicontinuous, see for instance
\cite[Chapter~II, Section~2]{Showalter1997MonotoneOperatorsBanachSpaceNonlinearPDE}, the convergence
\eqref{eq:ueps-to-u-XT} gives
\[
 \mathcal A_\kappa(u_\e)
 \rightharpoonup
 \mathcal A_\kappa(u)
 \quad\text{in }L^{p'}(0,T;W^{-1,p'}).
\]
Also, by Lemma~\ref{lem:Bnu-lip} and \eqref{eq:uepsbar-to-ubar},
\[
 \mathcal B_\nu\bar u_\e
 \to
 \mathcal B_\nu\bar u
 \quad\text{in }L^{p'}(0,T;W^{-1,p'}).
\]
By Lemma~\ref{lem:small-kernel-error},
\[
 u_\e'
 +
 \mathcal A_\kappa(u_\e)
 +
 \mathcal B_\nu\bar u_\e
 =
 f+E_\e ,
\]
and \(E_\e\to0\) in \(L^{p'}(0,T;W^{-1,p'})\). Passing to the limit gives
\[
 u'
 +
 \mathcal A_\kappa(u)
 +
 \mathcal B_\nu\bar u
 =
 f
 \quad\text{in }L^{p'}(0,T;W^{-1,p'}).
\]
Thus \(u\) is a weak solution of \eqref{eq:no-atom-problem} with initial
datum \(u_0\) and past history \(\phi\).

The estimate \eqref{eq:no-atom-estimate} follows from
Lemma~\ref{lem:uniform-truncation-estimate} and weak lower semicontinuity. Thus
\[
\begin{aligned}
 &\|u\|_{C([0,T];L^2)}^2
 +\|u\|_{L^p(0,T;W^{1,p}_0)}^p
 +\|u'\|_{L^{p'}(0,T;W^{-1,p'})}^{p'} \\
 \le{}&
 C\left(
 1+\|u_0\|_{L^2}^2
 +\|\phi\|_{L^p(-r,0;W^{1,p}_0)}^p
 +\|f\|_{L^{p'}(0,T;W^{-1,p'})}^{p'}
 \right).
\end{aligned}
\]

We finally show uniqueness. Let \(u\) and \(v\) be two weak solutions
with the same initial datum and the same past history. Set \(z:=u-v\).
Then \(z(0)=0\), and \(\bar u\) and \(\bar v\) have the same past history.
Subtracting the two equations gives
\[
 z'
 +
 \mathcal F_\mu\bar u
 -
 \mathcal F_\mu\bar v
 =
 0
 \quad\text{in }L^{p'}(0,T;W^{-1,p'}).
\]
Testing by \(z\), integrating over \((0,t)\), and using
Assumption~\ref{ass:pathwise-coercivity}, we obtain
\[
 \frac12\|z(t)\|_{L^2}^2
 +
 \gamma\int_0^t
 \|z(\tau)\|_{W^{1,p}_0}^p\,d\tau
 \le0 .
\]
Hence \(z=0\) on \([0,T]\), and uniqueness follows. The proof is complete.
\end{proof}

For the proof of the stability theorem, we first record a convergence lemma.
The lemma treats the atom at the origin and the residual part as one total kernel.

\begin{lemma}\label{lem:fixed-path-total-kernel}
Let \(\mu_n,\mu\in\M([0,r])\), with
\(\mu_n\stackrel{*}{\rightharpoonup}\mu\) in \(\M([0,r])\), and suppose
that \(\{\mu_n\}\) is bounded in \(\M([0,r])\). Put
\[
\begin{aligned}
 \kappa_n&:=\mu_n(\{0\}),&
 \nu_n&:=\mu_n-\kappa_n\delta_0,\\
 \kappa&:=\mu(\{0\}),&
 \nu&:=\mu-\kappa\delta_0 .
\end{aligned}
\]
If \(w_n\to w\) in \(L^p(-r,T;W^{1,p}_0)\), then
\begin{equation}
 \kappa_n\mathcal C(w_n)
 +\mathcal B_{\nu_n}w_n
 \to
 \kappa\mathcal C(w)
 +\mathcal B_\nu w
 \quad\text{in }L^{p'}(0,T;W^{-1,p'}).
 \label{eq:fixed-path-total-kernel}
\end{equation}
Here \(\mathcal C(w_n)\) denotes the function
\(t\mapsto\mathcal C(w_n(t))\) on \((0,T)\).
\end{lemma}

\begin{proof}
Put
\[
 X:=L^{p'}(0,T;W^{-1,p'}),
 \qquad
 M:=\sup_n\|\mu_n\|_{\M([0,r])}.
\]
For each \(n\), we write
\[
 \kappa_n\mathcal C(w_n)+\mathcal B_{\nu_n}w_n
 =
 \int_{[0,r]}\mathcal C(w_n(\cdot-s))\,d\mu_n(s)
 \quad\text{in }X.
\]
The same identity holds for \(\mu\) and \(w\).

We first estimate the part coming from \(w_n-w\). Since
\(w_n\to w\) in \(L^p(-r,T;W^{1,p}_0)\), the sequence is bounded there.
Choose \(R>0\) such that
\[
 \|w_n\|_{L^p(-r,T;W^{1,p}_0)}
 +\|w\|_{L^p(-r,T;W^{1,p}_0)}
 \le R
\]
for all \(n\). By \eqref{eq:C-local-lip}, after shifting in time,
\[
 \|\mathcal C(w_n(\cdot-s))-\mathcal C(w(\cdot-s))\|_X
 \le
 L_{(0,T),R}\|w_n-w\|_{L^p(-r,T;W^{1,p}_0)}
\]
for every \(s\in[0,r]\). Hence
\[
\begin{aligned}
&\left\|
 \int_{[0,r]}
 \{\mathcal C(w_n(\cdot-s))-\mathcal C(w(\cdot-s))\}
 \,d\mu_n(s)
\right\|_X \\
\le{}&
 M L_{(0,T),R}
 \|w_n-w\|_{L^p(-r,T;W^{1,p}_0)}
 \to0 .
\end{aligned}
\]

We now treat the remaining part, where \(w\) is fixed and the measure varies. Define
\[
 g(s):=\mathcal C(w(\cdot-s))\in X,
 \qquad 0\le s\le r .
\]
By \eqref{eq:C-local-lip}, we have \(g\in C([0,r];X)\). Approximate \(g\)
uniformly by finite sums
\[
 g_N(s)=\sum_{j=1}^N \chi_j(s)x_j,
 \qquad
 \chi_j\in C([0,r]),\quad x_j\in X .
\]
For such \(g_N\), the convergence follows from the scalar weak-star
convergence of \(\mu_n\) to \(\mu\). Passing to the uniform limit \(g_N\to g\), we obtain
\[
 \int_{[0,r]}g(s)\,d\mu_n(s)
 \to
 \int_{[0,r]}g(s)\,d\mu(s)
 \quad\text{in }X .
\]
Combining this convergence with the preceding estimate proves
\eqref{eq:fixed-path-total-kernel}. 
\end{proof}

We now prove the stability theorem. The key point is to compare \(u_n\) with
an auxiliary path having the same past history.

\begin{proof}[Proof of Theorem~\ref{thm:weakstar-stability}]
By Theorem~\ref{thm:extension-no-atom}, there exist unique weak solutions
\(u_n\) and \(u\). Define
\[
 \widetilde u_n(t):=
 \begin{cases}
  \phi_n(t), & -r<t<0,\\
  u(t), & 0\le t\le T.
 \end{cases}
\]
Then \(\bar u_n\) and \(\widetilde u_n\) have the same past history
\(\phi_n\), and
\[
 \widetilde u_n\to\bar u
 \quad\text{in }L^p(-r,T;W^{1,p}_0).
\]
Set \(z_n:=u_n-u\) and define
\[
 G_n:=f_n-f-
 \left(
  \mathcal F_{\mu_n}\widetilde u_n-
  \mathcal F_\mu\bar u
 \right).
\]
Then the equations for \(u_n\) and \(u\) give
\begin{equation}
 z_n'
 +\mathcal F_{\mu_n}\bar u_n
 -\mathcal F_{\mu_n}\widetilde u_n
 =G_n .
 \label{eq:weakstar-difference-equation}
\end{equation}
We claim that
\begin{equation}
 G_n\to0
 \quad\text{in }L^{p'}(0,T;W^{-1,p'}).
 \label{eq:Gn-to-zero}
\end{equation}
The convergence \(f_n\to f\) is assumed. Since \(\widetilde u_n(t)=\bar u(t)=u(t)\) for \(0<t<T\), the definition
\(\mathcal A_\kappa=\mathcal A_{\rm raw}+\kappa\mathcal C\) gives
\[
\begin{aligned}
 \mathcal F_{\mu_n}\widetilde u_n-
 \mathcal F_\mu\bar u
 =
 \kappa_n\mathcal C(\widetilde u_n)
 +\mathcal B_{\nu_n}\widetilde u_n 
 -
 \kappa\mathcal C(\bar u)
 -\mathcal B_\nu\bar u .
\end{aligned}
\]
By Lemma~\ref{lem:fixed-path-total-kernel}, this term converges to zero in
\(L^{p'}(0,T;W^{-1,p'})\). Thus \eqref{eq:Gn-to-zero} holds.

Since \(\bar u_n\) and \(\widetilde u_n\) have the same past history,
Assumption~\ref{ass:pathwise-coercivity} applied to \(\mu_n\) gives, after
testing \eqref{eq:weakstar-difference-equation} by \(z_n\) and integrating
over \((0,t)\),
\[
\begin{aligned}
 &\frac12\|z_n(t)\|_{L^2}^2
 +\gamma\int_0^t\|z_n(\tau)\|_{W^{1,p}_0}^p\,d\tau\\
 \le{}&
 \frac12\|u_{0,n}-u_0\|_{L^2}^2 
 +
 \int_0^t\dual{G_n(\tau)}{z_n(\tau)}\,d\tau .
\end{aligned}
\]
Young's inequality yields
\[
 \|z_n\|_{C([0,T];L^2)}^2
 +\|z_n\|_{L^p(0,T;W^{1,p}_0)}^p
 \le
 C\left(
  \|u_{0,n}-u_0\|_{L^2}^2
  +\|G_n\|_{L^{p'}(0,T;W^{-1,p'})}^{p'}
 \right).
\]
The right-hand side tends to zero by \(u_{0,n}\to u_0\) and
\eqref{eq:Gn-to-zero}. Therefore
\[
 u_n\to u
 \quad\text{in }\mathcal X_T.
\]
The proof is complete.
\end{proof}

\section{Examples and verification}
\label{sec:examples}

In this final section, we verify the abstract hypotheses for two classes of
Laplacian-type delay terms. The first class consists of ordinary
\(p\)-Laplacian delays with residual kernels separated from the origin.
The second class treats residual kernels that may reach the origin, including
collapsing delayed atoms, in a regularized setting and under a smallness
condition on the residual parts of the kernels.

We begin with the separated case.

\begin{example}
\label{ex:p-laplace-separated}
Let \(p\ge2\), \(a>0\), and set
\[
 \mathcal C_p(u):=
 -\diver\bigl(|\nabla u|^{p-2}\nabla u\bigr),
 \qquad
 \mathcal A_0(u):=a\mathcal C_p(u).
\]
We take \(I=\{0\}\) and \(\kappa=0\). Let
\(\rho\in\M([0,r])\) satisfy
\[
 \supp|\rho|\subset[\tau_*,r]
\]
for some \(\tau_*>0\). Then the equation is
\[
 u'(t)+a\mathcal C_p(u(t))
 +\int_{(0,r]}\mathcal C_p(\bar u(t-s))\,d\rho(s)=f(t).
\]

For \(v\in W^{1,p}_0\), H\"older's inequality gives
\[
\begin{aligned}
 \left|\dual{\mathcal C_p(u)}{v}\right|
 &=
 \left|
 \int_\Omega |\nabla u|^{p-2}\nabla u\cdot\nabla v\,dx
 \right|  \\
 &\le
 \bigl\||\nabla u|^{p-1}\bigr\|_{L^{p'}}
 \|\nabla v\|_{L^p}  \\
 &=
 \|\nabla u\|_{L^p}^{p-1}\|\nabla v\|_{L^p}
 \le
 C\|u\|_{W^{1,p}_0}^{p-1}\|v\|_{W^{1,p}_0}.
\end{aligned}
\]
Taking the supremum over \(\|v\|_{W^{1,p}_0}\le1\), we obtain
\[
 \|\mathcal C_p(u)\|_{W^{-1,p'}}
 \le
 C\bigl(1+\|u\|_{W^{1,p}_0}^{p-1}\bigr).
\]
Let \(J\subset\R\) be a bounded interval and let
\(w,z\in L^p(J;W^{1,p}_0)\). For fixed \(t\in J\) and
\(\psi\in W^{1,p}_0\), we estimate
\[
\begin{aligned}
&\left|
 \dual{\mathcal C_p(w(t))-\mathcal C_p(z(t))}{\psi}
\right|  \\
={}&
\left|
 \int_\Omega
 \bigl(
 |\nabla w(t)|^{p-2}\nabla w(t)
 -
 |\nabla z(t)|^{p-2}\nabla z(t)
 \bigr)\cdot\nabla\psi\,dx
\right|  \\
\le{}&
 C\int_\Omega
 \bigl(|\nabla w(t)|+|\nabla z(t)|\bigr)^{p-2}
 |\nabla(w-z)(t)|\,|\nabla\psi|\,dx  \\
\le{}&
 C
 \bigl(
  \|w(t)\|_{W^{1,p}_0}
  +\|z(t)\|_{W^{1,p}_0}
 \bigr)^{p-2}
 \|w(t)-z(t)\|_{W^{1,p}_0}
 \|\psi\|_{W^{1,p}_0}.
\end{aligned}
\]
Taking the supremum over \(\|\psi\|_{W^{1,p}_0}\le1\), we get
\[
\begin{aligned}
&\|\mathcal C_p(w(t))-\mathcal C_p(z(t))\|_{W^{-1,p'}}  \\
\le{}&
 C
 \bigl(
  \|w(t)\|_{W^{1,p}_0}
  +\|z(t)\|_{W^{1,p}_0}
 \bigr)^{p-2}
 \|w(t)-z(t)\|_{W^{1,p}_0}.
\end{aligned}
\]
Applying H\"older's inequality,
we obtain
\[
\begin{aligned}
&\|\mathcal C_p(w)-\mathcal C_p(z)\|_{L^{p'}(J;W^{-1,p'})}  \\
\le{}&
 C
 \left(
  \|w\|_{L^p(J;W^{1,p}_0)}
  +\|z\|_{L^p(J;W^{1,p}_0)}
 \right)^{p-2}
 \|w-z\|_{L^p(J;W^{1,p}_0)} .
\end{aligned}
\]
The same estimate, applied to time-independent functions, gives the
continuity of \(\mathcal C_p:W^{1,p}_0\to W^{-1,p'}\). Hence
\(\mathcal C_p\) satisfies Assumption~\ref{ass:delayed-flux}.

Furthermore,
\[
\begin{aligned}
&\dual{\mathcal A_0(u)-\mathcal A_0(v)}{u-v}  \\
={}&
 a\int_\Omega
 \bigl(
 |\nabla u|^{p-2}\nabla u
 -
 |\nabla v|^{p-2}\nabla v
 \bigr)\cdot\nabla(u-v)\,dx  \\
\ge{}&
c\|u-v\|_{W^{1,p}_0}^{p}.
\end{aligned}
\]
The growth estimate for \(\mathcal A_0\) follows from the estimate for
\(\mathcal C_p\):
\[
 \|\mathcal A_0(u)\|_{W^{-1,p'}}
 =
 a\|\mathcal C_p(u)\|_{W^{-1,p'}}
 \le
 C\bigl(1+\|u\|_{W^{1,p}_0}^{p-1}\bigr).
\]
Moreover, by Poincaré's inequality,
\[
 \dual{\mathcal A_0(u)}{u}
 =
 a\int_\Omega |\nabla u|^p\,dx
 \ge
 c a\|u\|_{W^{1,p}_0}^p .
\]
The hemicontinuity of \(\mathcal A_0\) is immediate from the continuity of
\(\xi\mapsto |\xi|^{p-2}\xi\). Thus \(\mathcal A_0\) satisfies
Assumption~\ref{ass:current}. Since \(\rho\) is separated from the origin, Theorem~\ref{thm:separated-wp}
implies that, for every admissible \(f,u_0,\phi\), the above problem has
a unique weak solution with initial datum \(u_0\) and past history
\(\phi\).
\end{example}

We next consider the case where residual kernels may reach the origin.
In the following regularized setting, Assumption~\ref{ass:pathwise-coercivity}
is verified under a smallness condition on the total variation norm of the
residual parts.

\begin{example}
\label{ex:collapsing-atom}
Let \(m>2\). Put \(q:=2m-2\) and \(q':=q/(q-1)\).
We apply the preceding results with \(q\) in place of \(p\), and write
\[
 \mathcal X_T^{(q)}
 :=
 C([0,T];L^2)\cap L^q(0,T;W^{1,q}_0).
\]
Set
\[
 \mathcal C_m(u):=
 -\diver\bigl(|\nabla u|^{m-2}\nabla u\bigr)
\]
and
\[
 \mathcal R_\varepsilon(u)
 :=
 -\diver\left(
 (\varepsilon^2+|\nabla u|^2)^{\frac{q-2}{2}}\nabla u
 \right),
 \qquad \varepsilon>0 .
\]
Fix \(a\in\R\), \(\beta>0\), and a bounded closed interval
\(I\subset[-a,\infty)\). For \(\kappa\in I\), set
\[
 \mathcal A_\kappa(u)
 :=
 \beta\mathcal R_\varepsilon(u)
 +(a+\kappa)\mathcal C_m(u).
\]
For a total kernel \(\mu\), set
\[
 \kappa_\mu:=\mu(\{0\}),
 \qquad
 \nu_\mu:=\mu-\kappa_\mu\delta_0,
\]
and assume \(\kappa_\mu\in I\). The corresponding equation is
\[
 u'(t)+\beta\mathcal R_\varepsilon(u(t))
 +(a+\kappa_\mu)\mathcal C_m(u(t))
 +(\mathcal B_{\nu_\mu}\bar u)(t)=f(t),
\]
where \(\mathcal B_{\nu_\mu}\) is defined with \(\mathcal C=\mathcal C_m\).

We first verify Assumption~\ref{ass:delayed-flux} for \(\mathcal C_m\).
Since \(q'(m-1)\le q\), the proof of Example~\ref{ex:p-laplace-separated}
gives the analogue of \eqref{eq:C-growth} with \(q\) in place of \(p\).
The same argument, using the pointwise Lipschitz estimate for the
\(m\)-Laplacian, gives the analogue of \eqref{eq:C-local-lip}. The continuity
of \(\mathcal C_m:W^{1,q}_0\to W^{-1,q'}\) follows in the same way. Thus
\(\mathcal C_m\) satisfies Assumption~\ref{ass:delayed-flux} with \(q\) in
place of \(p\).

Next, we use the inequality
\[
\begin{aligned}
&\left[
 (\varepsilon^2+|\xi|^2)^{\frac{q-2}{2}}\xi
 -
 (\varepsilon^2+|\eta|^2)^{\frac{q-2}{2}}\eta
\right]\cdot(\xi-\eta)  \\
\ge{}&
 c(\varepsilon^2+|\xi|^2+|\eta|^2)^{\frac{q-2}{2}}
 |\xi-\eta|^2 ,
\end{aligned}
\]
valid for \(\xi,\eta\in\R^d\). Applying this, we obtain
\[
\begin{aligned}
&\dual{
 \mathcal R_\varepsilon(u)-\mathcal R_\varepsilon(v)
}{u-v}  \\
\ge{}&
 c\int_\Omega
 (\varepsilon^2+|\nabla u|^2+|\nabla v|^2)^{\frac{q-2}{2}}
 |\nabla(u-v)|^2\,dx  \\
\ge{}&
 c\|u-v\|_{W^{1,q}_0}^{q}.
\end{aligned}
\]
The last inequality follows from
\[
 |\nabla(u-v)|^2
 \le
 2\bigl(|\nabla u|^2+|\nabla v|^2\bigr)
\]
and Poincaré's inequality.
Since \(\mathcal C_m\) is monotone and \(a+\kappa\ge0\), we obtain
\[
\begin{aligned}
&\dual{
 \mathcal A_\kappa(u)-\mathcal A_\kappa(v)
}{u-v}  \\
={}&
 \beta\dual{
 \mathcal R_\varepsilon(u)-\mathcal R_\varepsilon(v)
}{u-v}
 +(a+\kappa)
 \dual{
 \mathcal C_m(u)-\mathcal C_m(v)
}{u-v}  \\
\ge{}&
 c\|u-v\|_{W^{1,q}_0}^{q}.
\end{aligned}
\]
Moreover,
\[
 \|\mathcal R_\varepsilon(u)\|_{W^{-1,q'}}
 \le
 C\bigl(1+\|u\|_{W^{1,q}_0}^{q-1}\bigr).
\]
Together with the estimate for \(\mathcal C_m\), and since \(I\) is bounded,
this gives
\[
 \|\mathcal A_\kappa(u)\|_{W^{-1,q'}}
 \le
 C\bigl(1+\|u\|_{W^{1,q}_0}^{q-1}\bigr)
\]
uniformly in \(\kappa\in I\). Also,
\[
\begin{aligned}
 \dual{\mathcal A_\kappa(u)}{u}
 &=
 \beta\dual{\mathcal R_\varepsilon(u)}{u}
 +(a+\kappa)\dual{\mathcal C_m(u)}{u}  \\
 &\ge
 c\|u\|_{W^{1,q}_0}^{q}.
\end{aligned}
\]
The hemicontinuity follows from the continuity of the corresponding vector
fields. Thus \(\mathcal A_\kappa\) satisfies
Assumption~\ref{ass:current} with \(q\) in place of \(p\).

We next verify Assumption~\ref{ass:pathwise-coercivity} for small residual kernels. 
There exists \(m_0>0\), depending on \(m,\beta,\varepsilon\) and
\(\Omega\), such that the condition holds for every family
\(\KM\subset\M([0,r])\) satisfying
\[
 \kappa_\mu\in I\quad(\mu\in\KM),
 \qquad
 \sup_{\mu\in\KM}\|\nu_\mu\|_{\M([0,r])}\le m_0 .
\]
Indeed, let \(y=w-z\). For \(\tau\ge0\), set
\[
 E(\tau):=
 \int_\Omega
 (\varepsilon^2+|\nabla w(\tau)|^2+|\nabla z(\tau)|^2)^{\frac{q-2}{2}}
 |\nabla y(\tau)|^2\,dx ,
\]
and put \(E(\tau)=0\) for \(\tau<0\). Then
\[
 \dual{
 \mathcal R_\varepsilon(w(\tau))
 -\mathcal R_\varepsilon(z(\tau))
 }{y(\tau)}
 \ge cE(\tau).
\]
Moreover, \(E(\tau)\) controls both
\(\|\nabla y(\tau)\|_{L^2}^2\) and
\(\|y(\tau)\|_{W^{1,q}_0}^q\).

For the term containing \(\mathcal B_{\nu_\mu}\), the pointwise estimate for
the \(m\)-Laplacian and Young's inequality give
\[
\begin{aligned}
&\left|
\dual{
 \mathcal C_m(w(\tau-s))-\mathcal C_m(z(\tau-s))
}{
 y(\tau)
}
\right|  \\
&\quad\le
 C\bigl(E(\tau-s)+E(\tau)\bigr).
\end{aligned}
\]
Here \(E(\tau-s)=0\) for \(\tau-s<0\), because the two paths have the same
history. Hence
\[
\begin{aligned}
&\int_0^t
 \left|
 \dual{
  (\mathcal B_{\nu_\mu}w)(\tau)-(\mathcal B_{\nu_\mu}z)(\tau)
 }{
  y(\tau)
 }
 \right|\,d\tau  \\
\le{}&
 C\|\nu_\mu\|_{\M([0,r])}
 \int_0^t E(\tau)\,d\tau .
\end{aligned}
\]
Combining this with the estimate for \(\mathcal R_\varepsilon\) and using
\(a+\kappa_\mu\ge0\), we get
\[
\begin{aligned}
&\int_0^t
 \dual{
 \mathcal F_\mu w(\tau)-\mathcal F_\mu z(\tau)
 }{
  y(\tau)
 }\,d\tau  \\
\ge{}&
 \left(c-C\|\nu_\mu\|_{\M([0,r])}\right)
 \int_0^t E(\tau)\,d\tau .
\end{aligned}
\]
Choosing \(m_0>0\) sufficiently small gives
\[
\int_0^t
 \dual{
 \mathcal F_\mu w(\tau)-\mathcal F_\mu z(\tau)
 }{
  y(\tau)
 }\,d\tau
 \ge
 \gamma\int_0^t \|y(\tau)\|_{W^{1,q}_0}^{q}\,d\tau
\]
for some \(\gamma>0\), uniformly in \(\mu\in\KM\). Thus
Assumption~\ref{ass:pathwise-coercivity} holds.

We now consider collapsing delayed atoms. Let
\[
 \kappa_n\to\kappa_0,\qquad c_n\to c,\qquad \sigma_n\to0,
 \qquad 0<\sigma_n\le r,
\]
and let \(\nu\in\M([0,r])\) satisfy \(|\nu|(\{0\})=0\). Set
\[
 \mu_n:=\kappa_n\delta_0+c_n\delta_{\sigma_n}+\nu,
 \qquad
 \mu:=(\kappa_0+c)\delta_0+\nu .
\]
Then \(\mu_n\stackrel{*}{\rightharpoonup}\mu\) in \(\M([0,r])\), and
\(\sup_n\|\mu_n\|_{\M([0,r])}<\infty\). Assume that \(\kappa_n\in I\) for every \(n\), that
\(\kappa_0+c\in I\), and that the family
\[
 \KM:=\{\mu\}\cup\{\mu_n:n\in\mathbb N\}
\]
satisfies the above smallness condition. For initial data, past histories,
and forcing terms satisfying the convergence hypotheses of
Theorem~\ref{thm:weakstar-stability}, that theorem, with \(q\) in place of
\(p\), implies that the corresponding solutions of
\[
\begin{aligned}
 &u_n'(t)+\beta\mathcal R_\varepsilon(u_n(t))
 +(a+\kappa_n)\mathcal C_m(u_n(t))  \\
 +&c_n\mathcal C_m(\bar u_n(t-\sigma_n))
 +(\mathcal B_\nu\bar u_n)(t)=f_n(t)
\end{aligned}
\]
converge in \(\mathcal X_T^{(q)}\) to the solution of
\[
 u'(t)+\beta\mathcal R_\varepsilon(u(t))
 +(a+\kappa_0+c)\mathcal C_m(u(t))
 +(\mathcal B_\nu\bar u)(t)=f(t).
\]
Thus the delayed atom \(c_n\delta_{\sigma_n}\) is absorbed into the
present-time \(m\)-Laplacian coefficient in the limit.
\end{example}

\begin{remark}
The regularized term in Example~\ref{ex:collapsing-atom} is used to
verify Assumption~\ref{ass:pathwise-coercivity} for residual kernels
that may reach the origin. For the unregularized operator
\[
 \mathcal A_\kappa(u)=(a+\kappa)\mathcal C_m(u),
\]
the corresponding coercivity estimate is not verified here.
\end{remark}

\section*{Acknowledgments}
This work was supported in part by JSPS KAKENHI Grant Numbers JP23H00085 and JP26K17020.

\end{document}